\documentclass[11pt,twoside,a4paper]{article}

\usepackage{graphicx,url}
\usepackage{etex}
\usepackage{setspace}
\usepackage{enumerate,lineno,setspace,float}
\usepackage[small,compact]{titlesec}
\usepackage{amsmath,amsfonts,amssymb,amsthm}
\usepackage{geometry}
\usepackage{fancyhdr}
\usepackage{amssymb,amsmath}
\usepackage{latexsym}
\usepackage{graphicx, float}
\usepackage{url}
\usepackage{ucs}
\usepackage{caption}
\usepackage{latexsym}
\usepackage{amssymb}
\usepackage{amsmath}
\usepackage{subcaption}
\newtheorem{Theorem}{Theorem}[section]
\newtheorem{Definition}[Theorem]{Definition}

\newtheorem{Lemma}[Theorem]{Lemma}
\newtheorem{Example}[Theorem]{Example}
\newtheorem{Observation}[Theorem]{Observation}

\usepackage{color}

\definecolor{Blue}{rgb}{0,0,1}
\definecolor{Red}{rgb}{1,0,0}

\long\def\delete#1{}

\input amssym.def
\input amssym.tex

\newcommand{\be}{\begin{equation}}
\newcommand{\ee}{\end{equation}}
\newcommand{\bea}{\begin{eqnarray}}
\newcommand{\eea}{\end{eqnarray}}
\newcommand{\bean}{\begin{eqnarray*}}
\newcommand{\eean}{\end{eqnarray*}}
\def\non{\nonumber}

\def\ra{\rangle}

\def\diam{{\rm diam}}

\def\span{{\rm span}}
\def\rn{{\rm rn}}

\def\ve{\varepsilon}

\def\ra{\rightarrow}

\def\({\left(}
\def\){\right)}
\def\[{\left[}
\def\]{\right]}

\begin{document}

\title{Optimal radio labelings of the Cartesian product of the generalized Peterson graph and tree}
\author{\textbf{Payal Vasoya} \\ Gujarat Technological University, Ahmedabad - 383 824, Gujarat (India) \\ E-mail : prvasoya92@gmail.com \\ \\
\textbf{Devsi Bantva} \\ Lukhdhirji Engineering College, Morvi - 363 642, Gujarat (India) \\ E-mail : devsi.bantva@gmail.com}

\pagestyle{myheadings}
\markboth{\centerline{Payal Vasoya and Devsi Bantva}}{\centerline{Optimal radio labelings of the Cartesian product of the generalized Peterson graph and tree}}
\date{}
\openup 0.8\jot
\maketitle

\begin{abstract}
A radio labeling of a graph $G$ is a function $f : V(G) \rightarrow \{0,1,2,\ldots\}$ such that $|f(u)-f(v)| \geq \diam(G) + 1 - d(u,v)$ for every pair of distinct vertices $u,v$ of $G$. The radio number of $G$, denoted by $\rn(G)$, is the smallest number $k$ such that $G$ has radio labeling $f$ with max$\{f(v):v \in V(G)\}$ = $k$. In this paper, we give a lower bound for the radio number for the Cartesian product of the generalized Petersen graph and tree. We present two necessary and sufficient conditions, and three other sufficient conditions to achieve the lower bound. Using these results, we determine the radio number for the Cartesian product of the Peterson graph and stars.

\end{abstract}

\section{Introduction}\label{Intro:sec}
The channel assignment problem, introduced by Hale in \cite{Hale}, deals with the task of assigning channels to the radio transmitters so that the interference constraints between two transmitters is satisfied and the span of channels used is kept to a minimum. The interference between two transmitters is primarily dependent on the proximity of transmitters. The closer the two transmitters are, higher the interference is and vice-versa. Initially, only two levels of interference were considered, namely minor and major, and accordingly two transmitters are classified as close and very close transmitters.

The problem of assignment of channels to transmitters is studied using a graph model in which transmitters are represented by vertices of a graph such that two vertices are adjacent if major interference occurs and at distance two if minor interference occurs between those two corresponding  transmitters. It is assumed that there is no interference if the vertices are distance two apart in the graph model. Roberts\cite{Roberts} suggested that a pair of transmitters having minor interference must receive different channels and a pair of transmitters having major interference must receive channels that are at least two apart. Motivated through this, Griggs and Yeh\cite{Griggs} introduced $L(2,1)$-labeling which is also known as distance two labeling as follows: An \emph{$L(2,1)$-labeling} (or distance two labeling) of a graph $G=(V(G),E(G))$ is a function $f : V(G) \rightarrow \{0,1,2,...\}$ such that for each pair of distinct vertices $u$ and $v$ satisfies the following conditions:
\begin{equation*}
 |f(u)-f(v)| \geq \left\{
\begin{array}{ll}
2, & \mbox{if  } d(u,v) = 1, \\
1, & \mbox{if  } d(u,v) = 2.
\end{array}
\right.
\end{equation*}
where $d(u,v)$ is the distance between $u$ and $v$ in $G$. The span of $f$ is defined as $\span(f) = \max\{|f(u)-f(v)|: u, v \in V(G)\}$ and the \emph{$\lambda$-number} (or the $\lambda_{2,1}$-number) of $G$, denoted by $\lambda(G)$ (or $\lambda_{2,1}(G)$), is the minimum span of an $L(2,1)$-labeling of $G$. An $L(2,1)$-labeling $f$ of $G$ is called optimal if $\span(f) = \lambda(G)$. The readers can refer the survey article \cite{Calamoneri} and \cite{Yeh} for recent results on $L(2,1)$-labeling of graph $G$.

It is observed that the interference among transmitters might go beyond two levels. A radio labeling extends the level of interference considered in $L(2,1)$-labeling from two to the largest possible - the diameter of G. Motivated through this, Chartrand et al.\cite{Chartrand1, Chartrand2} introduced the concept of
radio labeling of a graph as follows:
\begin{Definition}
{\em
A \emph{radio labeling} of a graph $G$ is a mapping $f: V(G) \rightarrow \{0, 1, 2, \ldots\}$ such that for every pair of distinct vertices $u, v$ of $G$,
$$
|f(u)-f(v)| \geq \diam(G) + 1 - d(u,v).
$$
The integer $f(u)$ is called the \emph{label} of $u$ under $f$, and the \emph{span} of $f$ is defined as
$$
\span(f) = \max\{|f(u)-f(v)|: u, v \in V(G)\}.
$$
The \emph{radio number} of $G$ is defined as
$$
\rn(G) = \min_{f} \span(f)
$$
with minimum taken over all radio labelings $f$ of $G$. A radio labeling $f$ of $G$ is called \emph{optimal} if $\span(f) = \rn(G)$. If $\diam(G)=2$ then radio labeling of $G$ is same as $L(2,1)$-labeling of $G$.
}
\end{Definition}

A radio labeling always assign label $0$ to some vertex then the maximum label used is a $\span(f)$. It is an  interesting and challenging task to find radio number for the graph. The radio labeling problem is well studied for trees by several authors; refer \cite{Bantva1,Tuza,Li,Daphne1,Liu}. The radio labeling is studied only for handful graph families other than trees. The radio number for all graphs of order $n$ and diameter $n-2$ is determined by Benson et al. in \cite{Benson}. Cada et al. studied the radio number of distance graph in \cite{Cada}. Saha and Panigrahi studied the radio number for toroidal grids and square of graphs in \cite{Saha1} and \cite{Saha2}, respectively. Bantva and Liu gave a lower bound for the radio number of block graphs and presented three necessary and sufficient condition to achieve the lower bound in \cite{Bantva4}. The same authors also discussed the radio number of line graph of trees which are block graphs. Das et al. gave a technique to find a lower bound for the radio number of graphs in \cite{Das}. In \cite{Zhou}, Zhou discussed the problem of radio channel assignment for the network modelled by Cayley graphs. Bantva and Liu discussed the radio number of Cartesian product of two trees in \cite{Bantva2022}.

The paper is structured as follows. In Section 2, we define necessary terms and notations required for the present work. In Section 3, we give a lower bound for the radio number of the Cartesian product of the generalized Peterson graph and a tree. We present two necessary and sufficient conditions and, three other sufficient conditions for a lower bound to be tight. Finally, using these results, we determine the radio number for the Cartesian product of the Peterson graph and stars.

\section{Preliminaries}\label{Preli:sec}
In this section, we define necessary terms and notations which will be used in the present work. We follow \cite{West} for standard graph theoretic definitions and notations. The distance between two vertices $u$ and $v$ in a graph $G$, denoted by $d(u,v)$, is the length of the shortest path joining  $u$ and $v$ in a graph $G$. The diameter of a graph $G$ is $\diam(G) = \max\{d(u,v) : u, v \in G\}$. The \emph{weight} of a graph $G$ from a vertex $u \in V(G)$ is defined as $w_G(u)=\sum_{v\in V(G)} d(u,v)$. A vertex $c\in V(G)$ is said to be \emph{weight center} of $G$ if $w_G(c)=\min\{w_G(u) : u \in V(G)\}$. The weight of  graph $G$, denoted by $w(G)$, is defined as $w(G)=w_G(c)$, where $c$ is the weight center of the graph $G$. Denote the set of all weight centers of graph $G$ by $W(G)$. A tree $T$ is a simple connected acyclic graph. The following is known about weight center(s) of a tree $T$.

\begin{Lemma}\cite{Liu}\label{lem:wt1} If $w$ is a weight center of a tree $T$. Then each component of $T-w$ contains at most $|V(T)|/2$ vertices.
\end{Lemma}

\begin{Lemma}\cite{Liu}\label{lem:wt2} Every tree $T$ has one or two weight centers, and $T$ has two weight centers, say, $W(T)=\{w,w'\}$ if and only if $ww'$ is an edge of $T$ and $T-ww'$ consists of two equal sized components.
\end{Lemma}

For a tree $T$ rooted at $r$, if $x$ is on the $r-y$ path then $x$ is called ancestor of $y$ and $y$ is called descendant of $x$. Note that the root of the tree $T$ is an ancestor of every vertex and every vertex is its own ancestor and descendant.

If $W(T)=\{w\}$ then we view a tree T rooted at $w$ and if $W(T)=\{w,w'\}$ then we view a tree $T$ rooted at both $w$ and $w'$. Let $u$ be any vertex adjacent to the weight center of $T$, then the subtree induced by $u$ with all its descendants is called a branch of a tree $T$. Two branches are called different if they induced by different vertices adjacent to the same weight center and opposite if they induced by different vertices adjacent to different weight centers. Two vertices $u$ and $v$ are said to be in different branches if either $u$ and $v$ are in different branches or one of them is a weight center of $T$ and other is in the branch of $T$ induced by the vertex adjacent to it. If $W(T)=\{w,w'\}$ then $u$ and $v$ are said to be in opposite branches if they belong to different components of $T-ww'$. Define a level function $L_T : V(T) \rightarrow \{0,1,2,\ldots \}$ by
$$L_T(x):=\min\{d(x,w):w\in W(T)\}.$$
The value of $L_T(x)$ is called level of the vertex $x$ in $T$. The total level of $T$, denoted by $L(T)$, is defined as $$L(T):=\sum_{x\in V(T)}L_T(x).$$
For any two vertices $x,y\in V(T)$, define $$\phi_T(x,y)=\max\{L_T(z):z \mbox{ is a common ancestor of } x \mbox{ and } y\},$$
$$\delta_T(x,y)=\left\{
\begin{array}{ll}
1, & \mbox{if $|W(T)|=2$ and, $x$ and $y$ are in opposite branches}, \\
0, & \mbox{otherwise}.
\end{array}
\right.$$
Define
$$\ve(T) = \left\{
\begin{array}{ll}
1, & \mbox{if $|W(T)|=1$}, \\[0.2cm]
0, & \mbox{if $|W(T)|=2$}.
\end{array}
\right.$$
\begin{Observation}\label{obs1}
Let $T$ be a tree and $x,y\in V(T)$, then
\begin{enumerate}[\rm (1)]
    \item $0\leq \phi_T(x,y)<\max\{L_T(v):v\in V(T)\}$,
    \item $\phi_T(x,y)=0$ if and only if $x$ and $y$ are in opposite or different branches,
    \item $d_T(x,y)=L_T(x)+L_T(y)+\delta_T(x,y)-2\phi_T(x,y)$.
\end{enumerate}
\end{Observation}

\begin{Definition} For integers $m \geq 3$ and $k \geq 1$, the generalized Petersen graph $P_{m,k}$ is the graph with vertex set $V(P_{m,k}) = \{u_i,v_i : i \in Z_m\}$ and edge set $E(P_{m,k}) = \{u_iu_{i+1}, u_iv_i, v_iv_{i+k} : i \in Z_m\}$.
\end{Definition}

\begin{Definition} Let $G$ and $H$ be any two simple connected graphs. The Cartesian product of $G$ and $H$ is denoted by $G\Box H$ with vertex set $V(G)\times V(H)$, where two vertices $(g,h)$ and $(g',h')$ are adjacent in $G\Box H$ if and only if $g$ and $g'$ are adjacent in $G$ and $h=h'$ in $H$; or $h$ and $h'$ are adjacent in $H$ and $g=g'$ in $G$.
\end{Definition}

\begin{Observation}\label{obs2}
Let $G$ be any graph of order $m$ with diameter $d_1$ and $H$ be any graph of order $n$ with diameter $d_2$, then
\begin{enumerate}[\rm (1)]
\item $|V(G\Box H)|=mn$,
\item $\diam(G\Box H)=d_1+d_2$,
\item For any two vertices $(g_1,h_1)$ and $(g_2,h_2)$ in $G \Box H$, $$d_{G \Box H}((g_1,h_1),(g_2,h_2))=d_G(g_1,g_2)+d_H(h_1,h_2).$$
\end{enumerate}
\end{Observation}

\section{Main results}
In this section, we continue to use terms and notations defined in previous section. Note that a radio labeling $f$ of graph $G$ induced an ordering $O(V(G)) := z_0,z_1,\ldots,z_{|V(G)|-1}$ of $V(G)$ such that
$$0 = f(z_0) < f(z_1) < \ldots < f(z_{p-1}) = \span(f).$$

\begin{Theorem}\label{thm:lower} Let $P_{m,k}$ be a generalized Petersen graph of order $2m$ and $T$ be a tree of order $n$. Denote $\diam(P_{m,k} \Box T) = d_G$, $\diam(P_{m,k}) = d_p$, $\diam(T) = d_t$ and $\ve(T) = \ve$. Then
\begin{equation}\label{rn:lower}
\rn(P_{m,k} \Box T) \geq (2mn-1)(d_t+\ve)-4m L(T).
\end{equation}
Moreover, the equality holds in \eqref{rn:lower} if and only if there exist an ordering $O(V(P_{m,k} \Box T)) := z_0, z_1, \ldots, z_{2mn-1}$ of $V(P_{m,k} \Box T)$, where $z_t = (x_{i_t},y_{j_t})$ such that the following all holds for all $0 \leq t \leq 2mn-2$:
\begin{enumerate}[\rm (a)]
\item $d_{P_{m,k}}(x_{i_t},x_{i_{t+1}}) = d_p$,
\item $L_T(y_{j_0})+L_T(y_{j_{2mn-1}}) = 0$,
\item $y_{j_t}$ and $y_{j_{t+1}}$ are in different branches when $|W(T)|=1$ and in opposite branches when $|W(T)|=2$,
\item the mapping $f$ defined on $O(V(P_{m,k} \Box T))$ by $f(z_0) = 0$ and $f(z_{t+1})= f(z_t)+d_t+\ve-L_T(y_{j_t})-L_T(y_{j_{t+1}})$ is a radio labeling of $P_{m,k} \Box T$.
\end{enumerate}
\end{Theorem}
\begin{proof} Denote $G = P_{m,k} \Box T$. It is suffices to prove that the span of any radio labeling $f$ of $G$ is no less than the right-hand side of \eqref{rn:lower}. Let $f$ be any radio labeling of $P_{m,k} \Box T$ with $0 = f(z_0) < f(z_1) < \ldots < f(z_{2mn-1}) = \span(f)$. Since $f$ is a radio labeling, we have $f(z_{i+1})-f(z_i) \geq d_G+1-d_G(z_i,z_{i+1}) = d_p+d_t+1-d_G(z_i,z_{i+1})$ for $0 \leq i \leq 2mn-2$. Summing up these $2mn-1$ inequalities, we obtain
\begin{equation}\label{eqn:span1}
\span(f) = f(z_{2mn-1}) \geq (2mn-1)(d_p+d_t+1)-\sum_{i=0}^{2mn-2}d_G(z_i,z_{i+1})
\end{equation}
Denote $z_t = (x_{i_t},y_{j_t})$ for $0 \leq t \leq 2mn-1$, where $x_{i_t} \in V(P_{m,k})$ and $y_{j_t} \in V(T)$. By Observations \ref{obs1} and \ref{obs2}, we have
\begin{eqnarray}
\sum_{t=0}^{2mn-2}d_G(z_t,z_{t+1}) & = & \sum_{t=0}^{2mn-2}[d_{P_{m,k}}(x_{i_t},x_{i_{t+1}})+d_T(y_{j_t},y_{j_{t+1}})] \non \\
& \leq & \sum_{t=0}^{2mn-2}[d_p + L_T(y_{j_t})+L_T(y_{j_{t+1}})+\delta_T(y_{j_t},y_{j_{t+1}})-2\phi_T(y_{j_t},y_{j_{t+1}})] \non \\
& = & (2mn-1)d_p+2\sum_{t=0}^{2mn-1}L_T(y_{j_t})-L_T(y_{j_0})-L_T(y_{j_{2mn-1}}) \non \\
& & \hspace{4cm}+\sum_{t=0}^{2mn-2}[\delta_T(y_{j_t},y_{j_{t+1}})-2\phi_T(y_{j_t},y_{j_{t+1}})] \non \\
& = & (2mn-1)d_p+4mL(T)-L_T(y_{j_0})-L_T(y_{j_{2mn-1}}) \non \\
& & \hspace{4cm}+\sum_{t=0}^{2mn-2}[\delta_T(y_{j_t},y_{j_{t+1}})-2\phi_T(y_{j_t},y_{j_{t+1}})].\label{eqn:span2}
\end{eqnarray}

We consider the following two cases.

\textsf{Case-1:} $|W(T)| = 1$.

In this case, $\delta_T(y_{j_t},y_{j_{t+1}}) = 0$, $\phi_T(y_{j_t},y_{j_{t+1}}) \geq 0$ for all $0 \leq t \leq 2mn-2$ and $L_T(y_{j_0})+L_T(y_{j_{2mn-1}}) \geq 0$. Substituting these all in \eqref{eqn:span2}, we get
\begin{equation}\label{eqn:span3}
\sum_{t=0}^{2mn-2}d_G(z_t,z_{t+1}) \leq (2mn-1)d_p+4mL(T).
\end{equation}
Substituting \eqref{eqn:span3} in \eqref{eqn:span1}, we obtain
\begin{equation*}
\span(f) = f(z_{2mn-1}) \geq (2mn-1)(d_t+1)-4m L(T).
\end{equation*}

\textsf{Case-2:} $|W(T)| = 2$.

In this case, $0 \leq \delta_T(y_{j_t},y_{j_{t+1}}) \leq 1$,  $\phi_T(y_{j_t},y_{j_{t+1}}) \geq 0$ for all $0 \leq t \leq 2mn-2$ and $L_T(y_{j_0})+L_T(y_{j_{2mn-1}}) \geq 0$. Substituting these all in \eqref{eqn:span2}, we get
\begin{equation}\label{eqn:span4}
\sum_{t=0}^{2mn-2}d_G(z_t,z_{t+1}) \leq (2mn-1)d_p+4mL(T)+(2mn-1).
\end{equation}
Substituting \eqref{eqn:span4} in \eqref{eqn:span2}, we obtain
$$
\span(f) = f(z_{2mn-1}) \geq (2mn-1)d_t-4mL(T).
$$
It is easy to see from above Cases 1 and 2 that the equality hold in \eqref{rn:lower} if and only if the conditions in moreover part hold which completes the proof.
\end{proof}

\begin{Theorem}\label{thm:main} Let $P_{m,k}$ be a generalized Petersen graph of order $2m$ and $T$ be a tree of order $n$. Denote $P_{m,k} \Box T = G$, $\diam(P_{m,k} \Box T) = d_G$, $\diam(P_{m,k}) = d_p$, $\diam(T) = d_t$ and $\ve(T) = \ve$. Then
\begin{equation}\label{rn:main}
\rn(P_{m,k} \Box T) = (2mn-1)(d_t+\ve)-4mL(T)
\end{equation}
if and only if there exist an ordering $O(V(P_{m,k} \Box T)) := z_0,z_1,\ldots,z_{2mn-1}$ of $V(P_{m,k} \Box T)$, where $z_t = (x_{i_t},y_{j_t})$ such that the following all holds for all $0 \leq t \leq 2mn-2$:
\begin{enumerate}[\rm (a)]
\item $d_{P_{m,k}}(x_{i_t},x_{i_{t+1}}) = d_p$,
\item $L_T(y_{j_0})+L_T(y_{j_{2mn-1}}) = 0$,
\item For any $z_a=(x_{i_a},y_{j_a})$ and $z_b = (x_{i_b},y_{j_b}),0 \leq a < b \leq 2mn-1$; the distance $d_{P_{m,k}}(x_{i_a},x_{i_b})$ and $d_T(y_{j_a},y_{j_b})$ satisfies
\begin{equation}\label{eqn:dab}
d_T(y_{j_a},y_{j_b})+d_{P_{m,k}}(x_{i_a},x_{i_b}) \geq \sum_{t=a}^{b-1}(L_T(y_{j_t})+L_T(y_{j_{t+1}})-d_t-\ve)+d_t+d_p+1.
\end{equation}
\end{enumerate}
\end{Theorem}
\begin{proof} \textsf{Necessity:} Suppose \eqref{rn:main} holds. By Theorem \ref{thm:lower}, there exist a radio labeling $f$ of $P_{m,k} \Box T$ with $0 = f(z_0) < f(z_1) < \ldots < f(z_{2mn-1}) = \rn(P_{m,k} \Box T)$, where $z_t = (x_{i_t},y_{j_t}), x_{i_t} \in V(P_{m,k})$ and $y_{j_t} \in V(T)$ such that the conditions (a)-(d) in Theorem \ref{thm:lower} hold. Hence, the conditions (a) and (b) are satisfied. For any two vertices $z_a$ and $z_b\;(0 \leq a < b \leq 2mn-1)$, by the condition (d) of Theorem \ref{thm:lower}, we have
\begin{equation*}
f(z_b)-f(z_a) = \sum_{t=a}^{b-1}(d_t+\ve-L_T(y_{j_t})-L_T(y_{j_{t+1}})).
\end{equation*}
Since $f$ is a radio labeling, $f(z_b)-f(z_a) \geq d_G+1-d_G(z_a,z_b) = d_p+d_t+1-d_T(y_{j_a},y_{j_b})-d_{P_{m,k}}(x_{i_a},x_{i_b})$. Substituting this in above equation, we obtain
\begin{equation*} \sum_{t=a}^{b-1}(d_t+\ve-L_T(y_{j_t})-L_T(y_{j_{t+1}})) = f(z_b)-f(z_a) \geq d_p+d_t+1-d_T(y_{j_a},y_{j_b})-d_{P_{m,k}}(x_{i_a},x_{i_b}).
\end{equation*}
Hence, the condition (c) is satisfied.

\textsf{Sufficiency:} Let an ordering $O(V(P_{m,k} \Box T)) := z_0,z_1,\ldots,z_{2mn-1}$ of $V(P_{m,k} \Box T)$ satisfying the conditions (a)-(c) of hypothesis. We prove that \eqref{rn:main} holds and for this purpose, it is enough to show that the conditions (c) and (d) in Theorem \ref{thm:lower} hold. Take $a=t$ and $b = t+1$ in \eqref{eqn:dab}, we have $d_T(y_{j_t},y_{j_{t+1}}) \geq L_T(y_{j_t})+L_T(y_{j_{t+1}})+1-\ve$ which implies that $y_{j_t}$ and $y_{j_{t+1}}$ are in different branches when $|W(T)|=1$ and opposite branches when $|W(T)|=2$. Hence, the condition (c) in Theorem \ref{thm:lower} is satisfied. Let $f$ is defined by $f(z_0)=0$ and $f(z_{t+1}) = f(z_t)+d_t+\ve-L_T(y_{j_t})-L_T(y_{j_{t+1}})$ for $0 \leq t \leq 2mn-2$. We prove that $f$ is a radio labeling. Let $z_a$ and $z_b\;(0 \leq a < b \leq 2mn-1)$ be two arbitrary vertices. Then
\begin{eqnarray*}
f(z_b)-f(z_a) & = & \sum_{t=a}^{b-1}(d_t+\ve-L_T(y_{j_t})-L_T(y_{j_{t+1}})) \\
& \geq & d_p+d_t+1-d_{P_{m,k}}(x_{i_a},x_{i_{b}})-d_T(y_{j_a},y_{j_{b}}) \\
& = & d_G+1-d_G(z_a,z_b).
\end{eqnarray*}
Hence $f$ is a radio labeling. Thus the condition (d) of Theorem \ref{thm:lower} is satisfied. The proof is complete.
\end{proof}

In the next result, we give three sufficient conditions to achieve the lower bound for the radio number for the Cartesian product of the generalized Peterson graph $P_{m,k}$ and a tree.

\begin{Theorem}\label{thm:suf} Let $P_{m,k}$ be a generalized Peterson graph of order $2m$ and $T$ be a tree of order $n$. Denote $\diam(P_{m,k}) = d_p$, $\diam(T) = d_t$ and $\ve(T) = \ve$. Then
\begin{equation}\label{rn:suf}
\rn(P_{m,k} \Box T) = (2mn-1)(d_t+\ve)-4mL(T)
\end{equation}
if there exist an ordering $O(V(P_{m,k} \Box T)) := z_0,z_1,\ldots,z_{2mn-1}$ of $V(P_{m,k} \Box T)$, where $z_t = (x_{i_t},y_{j_t})$ such that the following holds for all $0 \leq t \leq 2mn-2$.
\begin{enumerate}[\rm (a)]
\item $d_{P_{m,k}}(x_{i_t},x_{i_{t+1}}) = d_p$ and for $b \geq a+2$,
$$d_{P_{m,k}}(x_{i_a},x_{i_b}) \geq \left\{
\begin{array}{ll}
d_p-(b-a-1), & \mbox{if $|W(T)|=1$}; \\[0.2cm]
d_p-(b-a-2), & \mbox{if $|W(T)|=2$}.
\end{array}
\right.$$
\item $L_T(y_{j_0})+L_T(y_{j_{2mn-1}})=0$,
\item $y_{j_t}$ and $y_{j_{t+1}}$ are in different branches when $|W(T)|=1$ and in opposite branches when $|W(T)|=2$,
\end{enumerate}
and one of the following holds.
\begin{enumerate}[\rm (d)]
\item[\rm (d)] $\min\{d_T(y_{j_t},y_{j_{t+1}}), d_T(y_{j_{t+1}},y_{j_{t+2}})\} \leq (d_t+1-\ve)/2$ for all $0 \leq t \leq 2mn-3$,
\item[\rm (e)] $d_T(y_{j_t},y_{j_{t+1}}) \leq (d_t+1+\ve)/2$ for all $0 \leq t \leq 2mn-2$,
\item[\rm (f)] For all $0 \leq t \leq 2mn-1$, $L_T(y_{j_t}) \leq (d_t+1)/2$ when $|W(T)|=1$ and $L_T(y_{j_t}) \leq (d_t-1)/2$ when $|W(T)|=2$ and for $z_a = (x_{i_a},y_{j_a})$, $z_b = (x_{i_b},y_{j_b})\;(0 \leq a < b \leq 2mn-1)$, if $y_{j_a}$ and $y_{j_b}$ are in the same branch of $T$ then $b-a \geq d_t+d_p$.
\end{enumerate}
\end{Theorem}
\begin{proof} We show that if (a)-(c) and one of (d)-(f) holds for an ordering $O(V(P_{m,k} \Box T)):=z_0,z_1,\ldots,z_{2mn-1}$ of $V(P_{m,k} \Box T)$, where $z_t = (x_{i_t},y_{j_t}); 0 \leq t \leq 2mn-1$ then (a)-(c) of Theorem \ref{thm:main} satisfies. Since the conditions (a) and (b) are identical in both Theorems \ref{thm:main} and \ref{thm:suf}, we verify the condition (c) of Theorem \ref{thm:main} only. Let $z_a = (x_{i_a},y_{j_a})$ and $z_b = (x_{i_b},y_{j_b}), 0 \leq a < b \leq 2mn-1$ be two arbitrary vertices. Denote the right-hand side of \eqref{eqn:dab} by $E(a,b)$. We consider the following two cases.

\textsf{Case-1:} $|W(T)| = 1$. In this case, recall that $\ve(T)=1$ and $\delta_T(y_{j_t},y_{j_{t+1}})=0$ for all $0 \leq t \leq 2mn-2$ by the definition of $\delta$.

\textsf{Subcase-1.1:} Suppose (a)-(d) holds. By (b) and (d), we have $L_T(y_{j_0}) = L_T(y_{j_{2mn-1}}) = 0 < d_t/2$ and for all $1 \leq t \leq 2mn-2$, $L_T(y_{j_t}) \leq \min\{L_T(y_{j_{t-1}})+L_T(y_{j_t}), L_T(y_{j_t})+L_T(y_{j_{t+1}})\} = \min\{d_T(y_{j_{t-1}},y_{j_t}),d_T(y_{j_t},y_{j_{t+1}})\} \leq d_t/2$.  If $z_a = (x_{i_a},y_{j_a})$ and $z_b = (x_{i_b},y_{j_b})$ such that $y_{j_a}$ and $y_{j_b}$ are in different branches then note that $d_T(y_{j_a},y_{j_b}) = L_T(y_{j_a})+L_T(y_{j_b})$. Hence, we have
\begin{eqnarray*}
E(a,b) & = & L_T(y_{j_a})+L_T(y_{j_b})+2\sum_{t=a+1}^{b-1}L_T(y_{j_t})-(b-a-1)d_t-(b-a)+d_p+1 \\
& \leq & L_T(y_{j_a})+L_T(y_{j_b})+2(b-a-1)(d_t/2)-(b-a-1)d_t+(d_p-b+a+1) \\
& = & L_T(y_{j_a})+L_T(y_{j_b})+(d_p-b+a+1) \leq d_T(y_{j_a},y_{j_b})+d_{P_{m,k}}(x_{i_a},x_{i_b}).
\end{eqnarray*}
If $z_a = (x_{i_a},y_{j_a})$ and $z_b = (x_{i_b},y_{j_b})$ such that $y_{j_a}$ and $y_{j_b}$ are in the same branch of $T$ then note that $b-a \geq 2$. If $b-a \geq 4$ then $\min\{L_T(y_{j_t})+L_T(y_{j_{t+1}}), L_T(y_{j_{t+1}})+L_T(y_{j_{t+2}})\} \leq d_t/2$ as $\min\{d_T(y_{j_t},y_{j_{t+1}}), d_T(y_{j_{t+1}},y_{j_{t+2}})\} \leq d_t/2$ and $\max\{L_T(y_{j_t})+L_T(y_{j_{t+1}}), L_T(y_{j_{t+1}})+L_T(y_{j_{t+2}})\} \leq d_t$. Hence, we have
\begin{eqnarray*}
E(a,b) & \leq & (2(d_t/2)+2d_t)-3d_t-4+d_p+1 = d_p-3 = d_p-b+a+1 \\
& \leq & d_{P_{m,k}}(x_{i_t},x_{i_{t+1}}) \leq d_{P_{m,k}}(x_{i_t},x_{i_{t+1}})+d_T(y_{j_t},y_{j_{t+1}}).
\end{eqnarray*}
If $b=a+3$ then there are two possibilities: (i) $d_T(y_{j_a},y_{j_{a+1}}) \leq d_t/2, d_T(y_{j_{a+1}},y_{j_{a+2}}) > d_t/2$ and $d_T(y_{j_{a+2}},y_{j_{a+3}}) \leq d_t/2$, (ii) $d_T(y_{j_a},y_{j_{a+1}}) > d_t/2, d_T(y_{j_{a+1}},y_{j_{a+2}}) \leq d_t/2$ and $d_T(y_{j_{a+2}},y_{j_{a+3}}) > d_t/2$. In case of (i), $L_T(y_{j_a})+L_T(y_{j_{a+1}}) \leq d_t/2, L_T(y_{j_{a+1}})+L_T(y_{j_{a+2}}) \leq d_t$ and $L_T(y_{j_{a+2}})+L_T(y_{j_{a+3}}) \leq d_t/2$. Hence,
 \begin{eqnarray*}
E(a,b) & \leq & (d_t/2+d_t+d_t/2)-2d_t-3+d_p+1 = d_p-2 \\
& \leq & d_{P_{m,k}}(x_{i_a},x_{i_b}) \leq d_{P_{m,k}}(x_{i_a},x_{i_b})+d_T(y_{j_a},y_{j_b}).
 \end{eqnarray*}
In case of (ii), we have $d_t/2 < L_T(y_{j_a})+L_T(y_{j_{a+1}}) \leq d_t, L_T(y_{j_{a+1}})+L_T(y_{j_{a+2}}) \leq d_t/2$ and $d_t/2 < L_T(y_{j_{a+2}})+L_T(y_{j_{a+3}}) \leq d_t$. Hence,
\begin{eqnarray*}
E(a,b) & = & L_T(y_{j_a})+L_T(y_{j_{a+3}})+2(L_T(y_{j_{a+1}})+L_T(y_{j_{a+2}}))-2d_t-3+d_p+1 \\
& \leq & L_T(y_{j_a})+L_T(y_{j_{a+3}})+2(d_t/2)-2d_t+d_p-2 \\
& = & L_T(y_{j_a})+L_T(y_{j_{a+3}})-2(d_t/2)+d_p-2 \\
& \leq & L_T(y_{j_a})+L_T(y_{j_{a+3}})-2\phi_T(y_{j_a},y_{j_{a+3}})+d_p-2 \\
& \leq & d_T(y_{j_a},y_{j_b})+d_{P_{m,k}}(x_{i_a},x_{i_b}).
\end{eqnarray*}
If $b=a+2$ then either (i) $d_T(y_{j_a},y_{j_{a+1}}) \leq d_t/2$ and $d_T(y_{j_{a+1}},y_{j_{a+2}}) \leq d_t/2$ or (ii) $d_T(y_{j_a},y_{j_{a+1}}) \leq d_t/2$ and $d_T(y_{j_{a+1}},y_{j_{a+2}}) > d_t/2$. In case of (i), $L_T(y_{j_a})+L_T(y_{j_{a+1}}) \leq d_t/2$ and $L_T(y_{j_{a+1}})+L_T(y_{j_{a+2}}) \leq d_t/2$ and hence, $E(a,b) \leq (d_t/2+d_t/2)-d_t-2+d_p+1 = d_p-1 \leq d_{P_{m,k}}(x_{i_a},x_{i_{a+2}}) \leq d_{P_{m,k}}(x_{i_a},x_{i_{a+2}})+d_T(y_{j_a},y_{j_{a+2}})$. In case of (ii), $L_T(y_{j_a})+L_T(y_{j_{a+1}}) \leq d_t/2$ and $L_T(y_{j_{a+1}})+L_T(y_{j_{a+2}}) > d_t/2$, that is $L_T(y_{j_{a+1}}) \leq d_t/2-L_T(y_{j_{a}})$. Hence,
\begin{eqnarray*}
E(a,b) & = & L_T(y_{j_a})+2L_T(y_{j_{a+1}})+L_T(y_{j_{a+2}})-d_t-2+d_p+1 \\
& \leq & L_T(y_{j_a})+2(d_t/2-L_T(y_{j_a}))+L_T(y_{j_{a+2}})-d_t+d_p-1 \\
& = & L_T(y_{j_a})+L_T(y_{j_{a+2}})-2L_T(y_{j_a})+d_p-1 \\
& \leq & L_T(y_{j_a})+L_T(y_{j_{a+2}})-2\phi_T(y_{j_a},y_{j_{a+2}})+d_p-1 \\
& = & d_T(y_{j_a},y_{j_{a+2}})+d_{P_{m,k}}(x_{i_a},x_{i_{i+2}}).
\end{eqnarray*}

\textsf{Subcase-1.2:} Suppose (a)-(c) and (e) hold. Then for all $0 \leq t \leq 2mn-2$, $L_T(y_{j_t})+L_T(y_{j_{t+1}}) \leq (d_t+2)/2$ as $d_T(y_{j_t},y_{j_{t+1}}) \leq (d_t+2)/2$ and, $y_{j_t}$ and $y_{j_{t+1}}$ are in different branches. If $b=a+1$ then $E(a,b) = L_T(y_{j_a})+L_T(y_{j_b})+d_p = d_T(y_{j_a},y_{j_b})+d_{P_{m,k}}(x_{i_a},x_{i_b})$. If $b=a+2$ then
\begin{eqnarray*}
E(a,b) & = & L_T(y_{j_a})+2L_T(y_{j_{a+1}})+L_T(y_{j_b})-d_t-1+d_p \\
& \leq & L_T(y_{j_a})+L_T(y_{j_b})+2((d_t+2)/2-L_T(y_{j_{a}}))-d_t-1+d_p \\
& = & L_T(y_{j_a})+L_T(y_{j_{b}})-2(L_T(y_{j_{a}})-1)+d_p-1 \\
& \leq & L_T(y_{j_a})+L_T(y_{j_{b}})-2\phi_T(y_{j_a},y_{j_b})+d_p-1 \\
& \leq & d_T(y_{j_a},y_{j_b})+d_{P_{m,k}}(x_{i_a},x_{i_b}).
\end{eqnarray*}
 If $b-a \geq 3$ then
\begin{eqnarray*}
E(a,b) & = & L_T(y_{j_a})+L_T(y_{j_b})+2\sum_{t=a+1}^{b-1}L_T(y_{j_t})-(b-a)(d_t+1)+d_t+d_p+1 \\
& \leq & L_T(y_{j_a})+L_T(y_{j_b})+2(b-a-1)((d_t+2)/2)-(b-a-1)(d_t+1)+d_p \\
& \leq & L_T(y_{j_a})+L_T(y_{j_b})+2((d_t+2)/2)-2d_t-2+d_p \\
& = & L_T(y_{j_a})+L_T(y_{j_b})-2((d_t-2)/2)+d_p-2 \\
& \leq & L_T(y_{j_a})+L_T(y_{j_b})-2\phi_T(y_{j_a},y_{j_b})+d_p-2 \\
& \leq & d_T(y_{j_a},y_{j_b})+d_{P_{m,k}}(x_{i_a},x_{i_b}).
\end{eqnarray*}

\textsf{Subcase-1.3:}  Suppose (a)-(c) and (f) hold. If $z_a = (x_{i_a},y_{j_a})$ and $z_b = (x_{i_b},y_{j_b})$ such that $y_{j_a}$ and $y_{j_b}$ are in different branches. Assume $d_t$ is even then $L_T(y_{j_t}) \leq d_t/2$ for all $0 \leq t \leq 2mn-1$. Hence, we obtain
\begin{eqnarray*}
E(a,b) & \leq & L_T(y_{j_a})+L_T(y_{j_b})+2\sum_{t=a+1}^{b-1}L_T(y_{j_t})-(b-a-1)(d_t+1)+d_p \\
& \leq & L_T(y_{j_a})+L_T(y_{j_b})+2\sum_{t=a+1}^{b-1}(d_t/2)-(b-a-1)(d_t+1)+d_p \\
& = & L_T(y_{j_a})+L_T(y_{j_b})+d_p-(b-a-1) \\
& \leq & d_T(y_{j_a},y_{j_b})+d_{P_{m,k}}(x_{i_a},x_{i_b}).
\end{eqnarray*}
Assume $d_t$ is odd then note that for any $(d_t+d_p)$ consecutive vertices in ordering $O(V(P_{m,k} \Box T)):=z_0,z_1,\ldots,z_{2mn-1}$, at most one $L(y_{j_q}) = (d_t+1)/2$ and for all other vertices $L(y_{j_t}) \leq (d_t-1)/2$. Hence, we have
\begin{eqnarray*}
E(a,b) & = & L_T(y_{j_a})+L_T(y_{j_b})+2\sum_{t=a+1}^{b-1}L_T(y_{j_t})-(b-a-1)(d_t+1)+d_p \\
& = & L_T(y_{j_a})+L_T(y_{j_b})+2(b-a-1)\(\frac{d_t-1}{2}\)+2\(\frac{b-a-1}{d_t+d_p}\)-(b-a-1)(d_t+1)+d_p \\
& = & L_T(y_{j_a})+L_T(y_{j_b})-(b-a-1)((d_t+d_p-2)/(d_t+d_p))+d_p-(b-a-1) \\
& \leq & L_T(y_{y_{j_a}})+L_T(y_{j_b})+d_p-(b-a-1) \\
& \leq & d_T(y_{j_a},y_{j_b})+d_{P_{m,k}}(x_{i_a},x_{i_b}).
\end{eqnarray*}
If $z_a = (x_{i_a},y_{j_a})$ and $z_b = (x_{i_b},y_{j_b})$ such that $y_{j_a}$ and $y_{j_b}$ are in the same branch of $T$ then note that $b-a \geq d_t+d_p$. Hence, $E(a,b) \leq L_T(y_{j_a})+L_T(y_{j_b})+d_p-(b-a-1) \leq L_T(y_{j_a})+L_T(y_{j_b})+d_p-(d_t+d_p-1) = L_T(y_{j_a})+L_T(y_{j_b})-2((d_t-1)/2) \leq L_T(y_{j_a})+L_T(y_{j_b})-2\phi_T(y_{j_a},y_{j_b}) \leq d_T(y_{j_a},y_{j_b}) \leq d_T(y_{j_a},y_{j_b})+d_{P_{m,k}}(x_{i_a},x_{i_b})$.

\textsf{Case-2:} $|W(T)|=2$. In this case, recall that $\ve(T)=0$ and $\delta_T(y_{j_t},y_{j_{t+1}})=1$ for all $0 \leq t \leq 2mn-2$ by the definition of $\delta$.

\textsf{Subcase-2.1:} Suppose (a)-(d) holds. By (b) and (d), we have $L_T(y_{j_0}) = L_T(y_{j_{2mn-1}}) = 0 < (d_t-1)/2$ and for all $1 \leq t \leq 2mn-2$, $L_T(y_{j_t}) \leq \min \{L_T(y_{j_{t-1}})+L_T(y_{j_t}), L_T(y_{j_t})+L_T(y_{j_{t+1}})\} = \min\{d_T(y_{j_{t-1}},y_{j_t})-1,d_T(y_{j_t},y_{j_{t+1}})-1\} \leq (d_t-1)/2$. If $z_a = (x_{i_a},y_{j_a})$ and $z_b = (x_{i_b},y_{j_b})$ such that $y_{j_a}$ and $y_{j_b}$ are in opposite branches of $T$ then $d_T(y_{j_a},y_{j_b}) = L_T(y_{j_a})+L_T(y_{j_b})+1$ and $y_{j_a}$ and $y_{j_b}$ are in different branches then $d_T(y_{j_a},y_{j_b}) = L_T(y_{j_a})+L_T(y_{j_b})$. Hence, in both these cases, we have
\begin{eqnarray*}
E(a,b) & = & L_T(y_{j_a})+L_T(y_{j_b})+2\sum_{t=a+1}^{b-1}L_T(y_{j_t})-(b-a-1)d_t+d_p+1 \\
& \leq & L_T(y_{j_a})+L_T(y_{j_b})+2(b-a-1)((d_t-1)/2)-(b-a-1)d_t+d_p+1 \\
& = & L_T(y_{j_a})+L_T(y_{j_b})+d_p-(b-a-2) \\
& \leq & d_T(y_{j_a},y_{j_b})+d_{P_{m,k}}(x_{i_a},x_{i_b}).
\end{eqnarray*}
If $z_a=(x_{i_a},y_{j_a})$ and $z_b=(x_{i_b},y_{j_b})$ such that $y_{j_a}$ and $y_{j_b}$ are in the same branch of $T$ then note that $b-a \geq 2$. If $b-a \geq 4$ then $\min\{L_T(y_{j_t})+L_T(y_{j_{t+1}}), L_T(y_{j_{t+1}})+L_T(y_{j_{t+2}})\} \leq (d_t-1)/2$ as $\min\{d_T(y_{j_t},y_{j_{t+1}})-1,d_T(y_{j_{t+1}},y_{j_{t+2}})-1\} \leq (d_t-1)/2$ and $\max\{L_T(y_{j_t})+L_T(y_{j_{t+1}}), L_T(y_{j_{t+1}})+L_T(y_{j_{t+2}})\} \leq d_t-1$. Hence, we have
\begin{eqnarray*}
E(a,b) & \leq & (2((d_t-1)/2)+2(d_t-1))-3d_t+d_p+1 \\
& = & d_p-2 \leq d_{P_{m,k}}(x_{i_a},x_{i_b})+d_T(y_{j_a},y_{j_b}).
\end{eqnarray*}
 If $b=a+3$ then there are two possibilities: (i) $d_T(y_{j_a},y_{j_{a+1}}) \leq (d_t+1)/2$, $d_T(y_{j_{a+1}},y_{j_{a+2}}) > (d_t+1)/2$ and $d_T(y_{j_{a+2}},y_{j_{a+3}}) \leq (d_t+1)/2$, (ii) $d_T(y_{j_a},y_{j_{a+1}}) > (d_t+1)/2$, $d_T(y_{j_{a+1}},y_{j_{a+2}}) \leq (d_t+1)/2$ and $d_T(y_{j_{a+2}},y_{j_{a+3}}) > (d_t+1)/2$. In case of (i), note that $L_T(y_{j_a})+L_T(y_{j_{a+1}}) \leq (d_t-1)/2$, $(d_t-1)/2 < L_T(y_{j_{a+1}})+L_T(y_{j_{a+2}}) \leq d_t-1$ and $L_T(y_{j_{a+2}})+L_T(y_{j_{a+3}}) \leq (d_t-1)/2$. Hence, $E(a,b) \leq ((d_t-1)/2+(d_t-1)+(d_t-1)/2)-2d_t+d_p+1 = d_p-1 \leq d_{P_{m,k}}(x_{i_a},x_{i_b}) \leq d_{P_{m,k}}(x_{i_a},x_{i_b})+d_T(y_{j_a},y_{j_b})$. In case of (ii), note that $(d_t-1)/2 < L_T(y_{j_a})+L_T(y_{j_{a+1}}) \leq d_t-1$, $L_T(y_{j_{a+1}})+L_T(y_{j_{a+2}}) \leq (d_t-1)/2$ and $(d_t-1)/2 < L_T(y_{j_{a+2}})+L_T(y_{j_{a+3}}) \leq d_t-1$. Hence,
\begin{eqnarray*}
E(a,b) & \leq & L_T(y_{j_a})+L_T(y_{j_b})+2(L_T(y_{j_{a+1}}) + L_T(y_{j_{a+2}}))-2d_t+d_p+1 \\
& \leq & L_T(y_{j_a})+L_T(y_{j_b})+2((d_t-1)/2)-2d_t+d_p+1 \\
& = & L_T(y_{j_a})+L_T(y_{j_b})-d_t+d_p \\
& = & L_T(y_{j_a})+L_T(y_{j_b})-2((d_t-1)/2)+d_p-1 \\
& \leq & L_T(y_{j_a})+L_T(y_{j_b})-2\phi_T(y_{j_a},y_{j_b})+d_p-1 \\
& \leq & d_{P_{m,k}}(x_{i_a},x_{i_b})+d_T(y_{j_a},y_{j_b}).
\end{eqnarray*}
If $b=a+2$ then either (i) $d_T(y_{j_a},y_{j_{a+1}}) \leq (d_t+1)/2$ and $d_T(y_{j_{a+1}},y_{j_{a+2}}) \leq (d_t+1)/2$, or (ii) $d_T(y_{j_a},y_{j_{a+1}}) \leq (d_t+1)/2$ and $d_T(y_{j_{a+1}},y_{j_{a+2}}) > (d_t+1)/2$. In case of (i), $L_T(y_{j_a})+L_T(y_{j_{a+1}}) \leq (d_t-1)/2$ and $L_T(y_{j_{a+1}})+L_T(y_{j_{a+2}}) \leq (d_t-1)/2$. Hence, we have $E(a,b) \leq ((d_t-1)/2+(d_t-1)/2)-d_t+d_p+1 = d_p \leq d_{P_{m,k}}(x_{i_a},x_{i_b})+d_T(y_{j_a},y_{j_b})$. In case of (ii), $L_T(y_{j_a})+L_T(y_{j_{a+1}}) \leq (d_t-1)/2$ and $(d_t-1)/2 < L_T(y_{j_{a+1}})+L_T(y_{j_{a+2}}) \leq d_t-1$, that is $L_T(y_{j_{a+1}}) \leq (d_t-1)/2-L_T(y_{j_a})$. Hence, $E(a,b) \leq L_T(y_{j_a})+2L_T(y_{j_{a+1}})+L_T(y_{j_{a+2}})-d_t+d_p+1 \leq L_T(y_{j_a})+L_T(y_{j_{a+2}})+2((d_t-1)/2-L_T(y_{j_a}))-d_t+d_p+1 = L_T(y_{j_a})+L_T(y_{j_{a+2}})-2L_T(y_{j_a})+d_p \leq L_T(y_{j_a})+L_T(y_{j_{a+2}})-2\phi_T(y_{j_a},y_{j_{a+2}})+d_p \leq d_{P_{m,k}}(x_{i_a},x_{i_b})+d_T(y_{j_a},y_{j_b})$.

\textsf{Subcase-2.2:} Suppose (a)-(c) and (e) holds. Then for all $0 \leq t \leq 2mn-2$, $L_T(y_{j_t})+L_T(y_{j_{t+1}}) \leq (d_t-1)/2$ as $d_T(y_{j_t},y_{j_{t+1}}) \leq (d_t+1)/2$ and, $y_{j_t}$ and $y_{j_{t+1}}$ are in the opposite branches. If $b=a+1$ then $E(a,b) = L_T(y_{j_a})+L_T(y_{j_b})+1+d_p = d_{P_{m,k}}(x_{i_a},x_{i_b})+d_T(y_{j_a},y_{j_b})$. If $b=a+2$ then
\begin{eqnarray*}
E(a,b) & = & L_T(y_{j_a})+2L_T(y_{j_{a+1}})+L_T(y_{j_b})-d_t+d_p+1 \\
& \leq & L_T(y_{j_a})+L_T(y_{j_b})+2((d_t-1)/2-L_T(y_{j_a}))+L_T(y_{j_b})-d_t+d_p+1 \\
& = & L_T(y_{j_a})+L_T(y_{j_b})-2L_T(y_{j_a})+d_p \\
& \leq & L_T(y_{j_a})+L_T(y_{j_b})-2\phi_T(y_{j_a},y_{j_b})+d_p \\
& \leq & d_T(y_{j_a},y_{j_b})+d_T(x_{i_a},x_{i_b}).
\end{eqnarray*}
If $b-a \geq 3$ then
\begin{eqnarray*}
E(a,b) & = & L_T(y_{j_a})+L_T(y_{j_b})+2\sum_{t=a+1}^{b-1}L_T(y_{j_t})-(b-a-1)d_t+d_p+1 \\
& \leq & L_T(y_{j_a})+L_T(y_{j_b})+2((d_t-1)/2)-2d_t+d_p+1 \\
& = & L_T(y_{j_a})+L_T(y_{j_b})-d_t+d_p \\
& = & L_T(y_{j_a})+L_T(y_{j_b})-2((d_t-1)/2)+d_p-1 \\
& \leq & L_T(y_{j_a})+L_T(y_{j_b})-2\phi_T(y_{j_a},y_{j_b})+d_p-1 \\
& \leq & d_T(y_{j_a},y_{j_b})+d_{P_{m,k}}(x_{i_a},x_{i_b}).
\end{eqnarray*}

\textsf{Subcase-2.3:} Suppose (a)-(c) and (f) holds. In this case, recall that $L_T(y_{j_t}) \leq (d_t-1)/2$ for all $0 \leq t \leq 2mn-1$. If $z_a = (x_{i_a},y_{j_a})$ and $z_b = (x_{i_b},y_{j_b})$ such that $y_{j_a}$ and $y_{j_b}$ are in opposite branches or in different branches of $T$ then
\begin{eqnarray*}
E(a,b) & \leq & L_T(y_{j_a})+L_T(y_{j_b})+2\sum_{t=a+1}^{b-1}L_T(y_{j_t})-(b-a-1)d_t+d_p+1 \\
& \leq & L_T(y_{j_a})+L_T(y_{j_b})+2(b-a-1)((d_t-1)/2)-(b-a-1)d_t+d_p+1 \\
& = & L_T(y_{j_a})+L_T(y_{j_b})+d_p-(b-a-2) \\
& \leq & d_T(y_{j_a},y_{j_b})+d_{P_{m,k}}(x_{i_a},x_{i_b}).
\end{eqnarray*}
If $z_a=(x_{i_a},y_{j_a})$ and $z_b=(x_{i_b},y_{j_b})$ such that $y_{j_a}$ and $y_{j_b}$ are in the same branch of $T$ then recall that $b-a \geq d_t+d_p$. Hence, as above, we have
\begin{eqnarray*}
E(a,b) & \leq & L_T(y_{j_a})+L_T(y_{j_b})+d_p-(b-a-2) \\
& \leq & L_T(y_{j_a})+L_T(y_{j_b})+d_p-(d_t+d_p-2) \\
& = & L_T(y_{j_a})+L_T(y_{j_b})-2((d_t-2)/2) \\
& \leq & L_T(y_{j_a})+L_T(y_{j_b})-2\phi_T(y_{j_a},y_{j_b}) \\
& \leq & d_T(y_{j_a},y_{j_b})+d_{P_{m,k}}(x_{i_a},x_{i_b}).
\end{eqnarray*}
\end{proof}


We now determine the radio number for the Cartesian product of the Petersen graph and star using our results.

The famous Petersen graph is the generalized Petersen graph $P_{5,2}$ as shown in figure \ref{Fig1}. We denote $V(P_{5,2}) = \{x_1,x_2,\ldots,x_{10}\}$ with $E(P_{5,2}) = \{x_ix_{i+2},x_jx_{j+3},x_1x_{10},x_2x_7,x_3x_9 : i=1,2,3,4,6,7,8; j=1,2,5,6,7\}$. An $n$-star is a tree consisting of $n$ leaves and another vertex joined to all leaves by edges denoted by $K_{1,n}$. Denote $V(K_{1,n})=\{y_j : j = 0,1,2,\ldots,n\}$ with $E(K_{1,n}) = \{y_0y_j : j = 1,2,\ldots,n\}$.

\begin{figure}[ht!]
\centering
\includegraphics[width=2in]{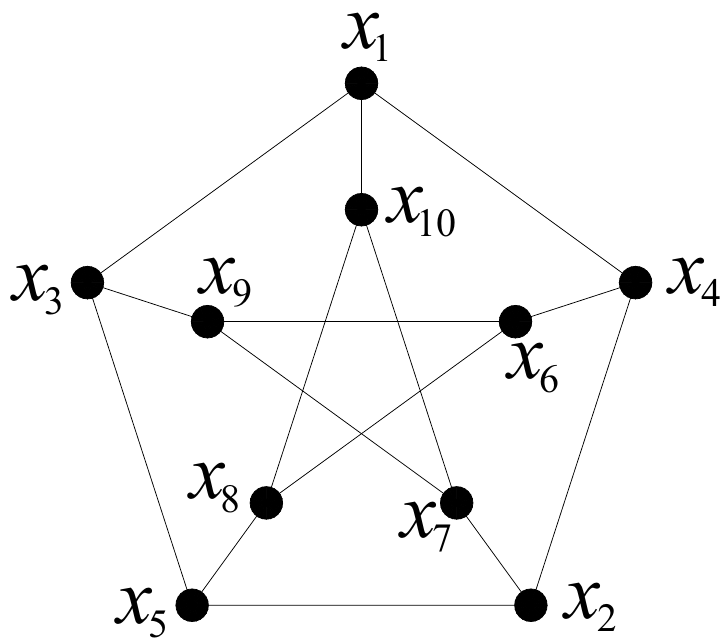}
\caption{The Peterson graph $P_{5,2}$.}
\label{Fig1}
\end{figure}

\begin{Theorem}\label{thm:PK1n} Let $P_{5,2}$ be the Petersen graph and $K_{1,n}$ be a star graph with $n\geq 3$. Then
\begin{equation}\label{rn:PK1n}
\rn(P_{5,2} \Box K_{1,n})= 10n+27.
\end{equation}
\end{Theorem}
\begin{proof} The order, diameter and the total level of $K_{1,n}$ are given by $n+1$, $2$ and $n$, respectively. Substituting all these in the right-hand side of \eqref{rn:lower}, we obtain the right-hand side of \eqref{rn:PK1n} which is a lower bound for $\rn(P_{5,2} \Box K_{1,n})$. Now we prove that this lower bound is tight. The result is easy to verify for $n=3,4$ and $5$. Hence, we consider $n \geq 6$ and for these cases, it suffices to give an ordering $O(V(P_{5,2} \Box K_{1,n})) := z_0,z_1,\ldots,z_{10n+9}$ of vertices of $P_{5,2} \Box K_{1,n}$ satisfying the conditions (a)-(c) of Theorem \ref{thm:main}. We use three permutations
$\alpha =
\begin{pmatrix}
1 & 2 & 3 & 4 & 5 & 6 & 7 & 8 & 9 & 10\\
1 & 5 & 9 & 3 & 7 & 2 & 4 & 6 & 8 & 10  \\
\end{pmatrix}
$,
$\beta =
\begin{pmatrix}
1 & 2 & 3 & 4 & 5 & 6 & 7 & 8 & 9 & 10\\
1 & 4 & 2 & 5 & 3 & 6 & 9 & 7 & 10 & 8  \\
\end{pmatrix}
$ and
$\tau =
\begin{pmatrix}
1 & 2 & 3 & 4 & 5 & 6 & 7 & 8 & 9 & 10\\
5 & 1 & 2 & 3 & 4 & 10 & 6 & 7 & 8 & 9  \\
\end{pmatrix}
$ to defined such ordering of $V(P_{5,2} \Box K_{1,n})$.
We consider the following two cases.

\textsf{Case-1:} $n \equiv 0$ (mod $5$).

In this case, we first rename $(x_i,y_j)\;(1 \leq i \leq 10, 0 \leq j \leq n)$ as $(a_r,b_s)$ as follows:
$$(a_r,b_s) = \left\{
\begin{array}{ll}
(x_i,y_j), & \mbox{if $1 \leq i \leq 10$ and $j=0$}, \\
(x_{\alpha\tau^{j-1}(i)},y_j), & \mbox{if $1 \leq i \leq 10$ and $1 \leq j \leq n$}.
\end{array}
\right.$$

We now define an ordering $O(V(P_{5,2} \Box K_{1,n})) := z_0,z_1,\ldots,z_{10n+9}$ as follows: Let $z_0 = (a_5,b_0)$ and for $1 \leq t \leq 10n+9$, let $z_t = (a_r,b_s)$, where
$$t := \left\{
\begin{array}{ll}
(r-1)n+s, & \mbox{if $1 \leq r \leq 10$ and $1 \leq s \leq n$}, \\
10n+r, & \mbox{if $1 \leq r \leq 4$ and $s=0$}, \\
10n+r-2, & \mbox{if $7 \leq r \leq 10$ and $s=0$}, \\
10n+9, & \mbox{if $r = 6$ and $s=0$}.
\end{array}
\right.$$

\textsf{Case-2:} $n \not\equiv 0$ (mod $5$).

In this case, we first rename $(x_i,y_j)\;(1 \leq i \leq 10, 0 \leq j \leq n)$ as $(a_r,b_s)$ as follows:
$$(a_r,b_s) = \left\{
\begin{array}{ll}
(x_{i},y_j), & \mbox{if $1 \leq i \leq 10$ and $j=0$}, \\
(x_{\tau^{j-1}(i)}, y_j), & \mbox{if $1 \leq i \leq 10, 1 \leq j \leq n$ and $n \equiv 1$ (mod $5$)}, \\
(x_{\beta\tau^{j-1}(i)}, y_j), & \mbox{if $1 \leq i \leq 10, 1 \leq j \leq n$ and $n \equiv 2$ (mod $5$)}, \\
(x_{\beta^3\tau^{j-1}(i)}, y_j), & \mbox{if $1 \leq i \leq 10, 1 \leq j \leq n$ and $n \equiv 3$ (mod $5$)}, \\
(x_{\beta^2\tau^{j-1}(i)}, y_j), & \mbox{if $1 \leq i \leq 10, 1 \leq j \leq n$ and $n \equiv 4$ (mod $5$)}, \\
\end{array}
\right.$$

We now define an ordering $O(V(P_{5,2} \Box K_{1,n})):=z_0,z_1,\ldots,z_{10n+9}$ as follows: Let $z_0=(a_5,b_0)$ and for $1 \leq t \leq 10n+9$, let $z_t=(a_r,b_s)$, where
$$t := \left\{
\begin{array}{ll}
(r-1)n+s, & \mbox{if $1 \leq r \leq 10$ and $1 \leq s \leq n$}, \\
10n+5-r, & \mbox{if $1 \leq r \leq 4$ and $s=0$}, \\
10n+r-1, & \mbox{if $6 \leq r \leq 10$ and $s=0$}.
\end{array}
\right.$$

Then, in both the cases above, note that $d_{P_{5,2}}(x_{i_t},x_{i_{t+1}}) = 2$ and $L_T(y_{j_0})+L_T(y_{j_{10n+9}}) = 0$, where $z_t = (x_{i_t},y_{j_t})$ for $0 \leq t \leq 10n+9$. Hence, the conditions (a)-(b) of Theorem \ref{thm:main} is satisfied.

\textsf{Claim:} The above defined ordering $O(V(P_{5,2} \Box K_{1,n})) := z_0,z_1,\ldots,z_{10n+9}$ satisfies \eqref{eqn:dab}.

Let $z_a = (x_{i_a},y_{j_a}), z_b = (x_{i_b},y_{j_b}); 0 \leq a < b \leq 10n+9$ be two vertices. Denote the right-hand side of \eqref{eqn:dab} by $E(a,b)$. If $b-a = 1$ then $E(a,b) = L_T(y_{j_a})+L_T(y_{j_b})-3+5 = L_T(y_{j_a})+L_T(y_{j_b}) + 2 = d_{P_{5,2}}(x_{i_a},x_{i_b})+d_{T}(y_{j_a},y_{j_b})$. Assume $b-a \geq 2$. If $a=0$ and $b=2$ then $d_{P_{m,k}}(x_{i_0},x_{i_2}) = d_T(y_{j_0},y_{j_2}) = 1$ and hence $E(a,b) = L_T(y_{j_0})+2L_T(y_{j_1})+L_T(y_{j_2})-d_t-1 = 2 = d_{P_{m,k}(x_{i_0},x_{i_2})}+d_T(y_{j_0},y_{j_2})$. If $a=0$ and $b \geq 3$ then $E(a,b) \leq 1 \leq d_{P_{m,k}}(x_{i_0},x_{i_b})+d_T(y_{j_0},y_{j_b})$ as $z_0 = (x_{i_0},y_{j_0}) \neq z_b = (x_{i_b},y_{j_b})$. Hence, the \eqref{eqn:dab} is satisfied. If $1 \leq a < b \leq 10n$ then observe that $L_T(y_{t}) = 1$ for all $1 \leq t \leq 10n$ and hence we have
\begin{eqnarray*}
E(a,b) & = & \sum_{t=a}^{b-1} [L_T(y_{j_t})+L_T(y_{j_{t+1}})-2-1] +2+2+1 \\
& \leq & 2(b-a)-3(b-a)+5 \\
& = & 5-(b-a).
\end{eqnarray*}
If $b-a = 2$ then note that $d_{P_{5,2}}(x_{i_a},x_{i_b}) \geq 1$ and $d_{T}(y_{j_a},y_{j_b}) = 2$ and hence $E(a,b) \leq 3 \leq d_{P_{5,2}}(x_{i_a},x_{i_b})+d_T(y_{j_a},y_{j_b})$. If $b-a = 3$ then observe that $d_{P_{5,2}}(x_{i_a},x_{i_b}) \geq 1$ and $d_T(y_{j_a},y_{j_b}) \geq 2$ and hence $E(a,b) \leq 2 \leq d_{P_{5,2}}(x_{i_a},x_{i_b})+d_T(y_{j_a},y_{j_b})$. If $b-a \geq 4$ then as $z_a \neq z_b$, we have $d_{P_{5,2}}(x_{i_a},x_{i_b})+d_T(y_{j_a},y_{j_b}) \geq 1$ and hence $E(a,b) \leq 1 \leq d_{P_{5,2}}(x_{i_a},x_{i_b})+d_T(y_{j_a},y_{j_b})$. Again if $10n < b \leq 10n+9$ then it is easy to verify that \eqref{eqn:dab} is satisfied. Therefore, the condition (c) of Theorem \ref{thm:main} is satisfied.
\end{proof}

\begin{Example} In the following Table \ref{tab:P52K16}, a vertex ordering $O(V(P_{5,2} \Box K_{1,6})):=z_0,z_1,\ldots,z_{69}$ and an optimal radio labeling of $P_{5,2} \Box K_{1,6}$ obtained from the proof of Theorem \ref{thm:PK1n} is shown.
\end{Example}
\begin{table}[ht!]
\centering
\caption{$\rn(P_{5,2} \Box K_{1,6}) = 87$.}\label{tab:P52K16}
\begin{tabular}{|c|l|l|l|l|l|l|l|l|l|l|l|}
\hline
$(x_i,y_j)$ & \multicolumn{1} {|c|}{$y_0$} & \multicolumn{1} {|c|}{$y_1$} & \multicolumn{1} {|c|}{$y_2$} & \multicolumn{1} {|c|}{$y_3$} & \multicolumn{1} {|c|}{$y_4$} & \multicolumn{1} {|c|}{$y_5$} & \multicolumn{1} {|c|}{$y_6$} \\ \hline
    $x_{1}$ & $z_{64} \ra 72$ & $z_{1} \ra 2$ &  $z_{26} \ra 27$ & $z_{21} \ra 22$ & $z_{16} \ra 17$ & $z_{11} \ra 12$ & $z_{6} \ra 7$  \\
    $x_2$ & $z_{63} \ra 69$ & $z_{7} \ra 8$ & $z_{2} \ra 3$ & $z_{27} \ra 28$ & $z_{22} \ra 23$ & $z_{17} \ra 18$ & $z_{12} \ra 13$  \\
    $x_3$ & $z_{62} \ra 66$ & $z_{13} \ra 14$ & $z_{8} \ra 9$ & $z_{3} \ra 4$ & $z_{28} \ra 29$ & $z_{23} \ra 24$ & $z_{18} \ra 19$  \\
    $x_4$ & $z_{61} \ra 63$ & $z_{19} \ra 20$ & $z_{14} \ra 15$ & $z_{9} \ra 10$ & $z_{4} \ra 5$ & $z_{29} \ra 30$ & $z_{24} \ra 25$  \\
    $x_5$ & $z_{0} \ra 0$ & $z_{25} \ra 26$ & $z_{20} \ra 21$ & $z_{15} \ra 16$ & $z_{10} \ra 11$ & $z_{5} \ra 6$ & $z_{30} \ra 31$  \\
    $x_6$ & $z_{65} \ra 75$ & $z_{31} \ra 32$ & $z_{56} \ra 57$ & $z_{51} \ra 52$ & $z_{46} \ra 47$ & $z_{41} \ra 42$ & $z_{36} \ra 37$  \\
    $x_7$ & $z_{66} \ra 78$ & $z_{37} \ra 38$ & $z_{32} \ra 33$ & $z_{57} \ra 58$ & $z_{52} \ra 53$ & $z_{47} \ra 48$ & $z_{42} \ra 43$  \\
    $x_8$ & $z_{67} \ra 81$ & $z_{43} \ra 44$ & $z_{38} \ra 39$ & $z_{33} \ra 34$ & $z_{58} \ra 59$ & $z_{53} \ra 54$ & $z_{48} \ra 49$  \\
    $x_9$ & $z_{68} \ra 84$ & $z_{49} \ra 50$ & $z_{44} \ra 45$ & $z_{39} \ra 40$ & $z_{34} \ra 35$ & $z_{59} \ra 60$ & $z_{54} \ra 55$  \\
    $x_{10}$ & $z_{69} \ra 87$ & $z_{55} \ra 56$ & $z_{50} \ra 51$ & $z_{45} \ra 46$ & $z_{40} \ra 41$ & $z_{35} \ra 36$ & $z_{60} \ra 61$  \\
    \hline
    \end{tabular}
\end{table}

\end{document}